\documentclass[10t,a4paper]{amsart}

\usepackage[english]{babel}
\usepackage[T1]{fontenc}
\usepackage[utf8]{inputenc}

\usepackage{amscd,amssymb,amsopn,amsmath,amsthm,graphics,amsfonts,enumerate,verbatim,calc}
\usepackage{amssymb, tikz}
\usepackage{mathtools}

\usepackage{hyperref}
\usepackage[capitalise]{cleveref}

\usepackage{setspace}
\usepackage{mathpazo}
\usepackage{color}
\usepackage{latexsym}
\usepackage{amsthm,amsfonts,amssymb,mathrsfs}
\usepackage{rotating}
\usepackage[leqno]{amsmath}
\usepackage{xspace}
\usepackage[all]{xy}
\usepackage{longtable}
\headsep=5 mm \numberwithin{equation}{section}
\newtheorem{theorem}{Theorem}[section]

\newtheorem{proposition}[theorem]{Proposition}
\newtheorem{corollary}[theorem]{Corollary}

\theoremstyle{definition}
\newtheorem{definition}[theorem]{Definition}
\theoremstyle{remark}
\newtheorem{remark}[theorem]{Remark}
\newtheorem{fact}[theorem]{Fact}
\newtheorem{example}[theorem]{Example}
\newtheorem{observation}[theorem]{Observation}

\newtheorem{question}[theorem]{Question}
\newtheorem{conjecture}[theorem]{Conjecture}

\DeclareMathOperator{\CI}{\textnormal{CI-dim}}

\newcommand{\im}{\operatorname{im}}

\newcommand{\Spec}{\operatorname{Spec}}

\newcommand{\G}{G}

\newcommand{\Der}{\operatorname{Der}}

\newcommand{\embdim}{\operatorname{embdim}}
\newcommand{\Ht}{\operatorname{ht}}
\newcommand{\pd}{\operatorname{p.dim}}

\newcommand{\Soc}{\operatorname{Soc}}
\newcommand{\Syz}{\operatorname{Syz}}

\newcommand{\Gdim}{\operatorname{Gdim}}

\newcommand{\id}{\operatorname{id}}
\newcommand{\Ext}{\operatorname{Ext}}

\newcommand{\Tor}{\operatorname{Tor}}

\newcommand{\Hom}{\operatorname{Hom}}

\newcommand{\depth}{\operatorname{depth}}

\newcommand{\Char}{\operatorname{char}}
\newcommand{\coker}{\operatorname{coker}}

\newcommand{\lo}{\longrightarrow}
\newcommand{\fm}{\frak{m}}
\newcommand{\fp}{\frak{p}}

\newcommand{\fa}{\frak{a}}

\newcommand{\fn}{\frak{n}}

\begin{document}

\author[]{Mohsen Asgharzadeh}

\address{}
\email{mohsenasgharzadeh@gmail.com}

\title[ ]
{A note on  the homology of differentials }


\begin{abstract} 
	We deal with the complete-intersection property of maximally differential ideals.
	Also, we connect the Gorenstein homology of derivations to the Gorenstein property of the base rings.
These equipped with some applications.
\end{abstract}

\maketitle

\section{Introduction}

Let $R$ be a commutative noetherian ring. In the case of  prime characteristic, each iteration of the Frobenius  map defines a new $R$-module structure on $R$, and this $R$-module is denoted by $ { ^n}R$. By Peskine-Szpiro functor
we mean $F^n(M)=M\otimes { ^n}R$.  Also, let $\mathcal{B}$ be a maximally differential ideal. By $G$-resolution of $\mathcal{B}$ we mean the Gorenstein resolution of $\mathcal{B}$, i.e., an acyclic complex of totally reflexive modules that  approximates $\mathcal{B}$. For more details, and for the unexplained definitions see \S2.
We start with:

\begin{question}(See \cite[Question 4.9]{int}) \label{Q1}
	\label{1}  If $R$ is Cohen-Macaulay and $T_\bullet$ is a minimal $G$-resolution of $\mathcal{B}$ such that $F_n(T_\bullet)$ is acyclic for all $n\geq 1$, then must $\mathcal{B}$ be a complete intersection?
\end{question}

In  \S 3  we give  negative answers to  Question \ref{Q1} from different point of views.
 Our list of examples contains zero and prime characteristic cases, zero and higher  dimension cases, and the case of integral domains. Our next goal is to understand the following amazing conjecture:

\begin{conjecture}(See  \cite[Conjecture 3.12]{int})\label{C1}
Suppose that $R$ is a local $k$-algebra for which the $R$-module $\Der_k(R)$ is finitely generated.
	\begin{enumerate}
	\item[a)]
If $\Gdim(\Der_k(R))<\infty$, then $R$ is Gorenstein;
\item[b)]
If $\CI(\Der_k(R))<\infty$, then $R$ is a complete intersection.
\end{enumerate}
\end{conjecture}

Recall that
Maloo \cite{mal} presented a nontrivial ring  $R$ such that $\Gdim(\Der_k(R))=\CI(\Der_k(R))=0<\infty$. Toward supporting his conjectures,  Miranda-Neto
asked: 

\begin{question}(See  \cite[Page 11]{int}) \label{Q2}
Is the Maloo's ring  Gorenstein? 
\end{question}

In \S4 we show Question \ref{Q2} has positive answer by proving that the proposed ring is complete-intersection.  Despite  this, we show Conjecture \ref{C1} is not true in the above 
setting.  Then, we try to rescue Conjecture \ref{C1}. In this regard, we present a number  of cases for which
the mentioned conjecture remains true. We do this 
by using different computational methods.
Also,  we present a tiny connection to the rigidity conjecture from deformation theory.

In Section 5 we study a connection between differentials and tensor products.
In this regard, we answer the following question:

\begin{question}(See  \cite[Question   3.7]{j}) Let $R$ be a local ring and let $M$ be a finite R-module. Does there exist a torsionless
$	R$-module $U \neq 0$ such that$$\depth(M  \otimes_R U )= \depth(N  \otimes_R U )$$
whenever $N$ is a non-zero finite $R$-module with $\depth M \leq \depth N$?
\end{question}
For a motivation, there is a connection from Question 1.4 to  a famous conjecture of  Herzog and Vasconcelos on the homology of differentials, see \cite[\S3]{j}.
We close \S5 by presenting additional observations on tensor products, higher $\Ext $ and $\Tor$-modules of differentials. This has an application.

\section{Preliminary definitions}

\begin{definition}
	We say that the ideal $\fa$ is  complete intersection if either $\fa=0$ or  $\fa$ can be generated by an $R$-sequence. 
\end{definition}
Here, $(-)^\ast$ means $\Hom_R(-,R)$.
\begin{definition}
	An $R$-module $M$ is called \emph{totally reflexive} provided that: 	\begin{enumerate}
		\item[i)] the natural map $M\rightarrow M^{**}$ is an isomorphism,
		\item[ii)] $\Ext^i_R(M,R)=\Ext^i_R(M^*,R)=0$ for all $i\geq 1$.
	\end{enumerate}  The \emph{Gorenstein dimension} of $M$, denoted $\Gdim_R(M)$, is defined to be the infimum of all nonnegative integers $n$, such that there exists an exact sequence  $0\rightarrow G_n\rightarrow\cdots\rightarrow G_0\rightarrow  M \rightarrow 0,$  in which each $G_i$ is a totally reflexive $R$-module. We call such a resolution, by a minimal  $\G$-resolution.
\end{definition}

Every finitely generated module over a Gorenstein ring has finite Gorenstein dimension. Moreover, if $R$ is local and $\Gdim_R(M)<\infty$, then it follows that $\Gdim_R(M)=\depth R-\depth_R(M)$, and we call it Auslander-Bridger formula. Also, if $(R,\fm,k)$ is local and that  $\Gdim_R(k)<\infty$, then it follows that $R$ is Gorenstein, and we call it the local-global principal. To see more information we recommend the  monograph \cite{AB}.

\begin{definition} 
i)	A derivation is an additive map $D:R\to R$ such that satisfies $D(ab) =aD(b) +bD(a)$.

ii) If $R$ is a $k$-algebra, the set of all derivations vanishes at $k$
is denoted by $\Der_k(R)$.

iii)  Let $\emptyset\neq\mathcal{D}$ be a  set of derivations.  An ideal $I$ is said to be $\mathcal{D}$-differential if $D(I)\subset I$ for all $D\in\mathcal{D}$.
\end{definition}

By noetherian condition, there  is a maximal $\mathcal{D}$-differential proper ideal with respect to the inclusion.  If  the ring is local, there is exactly one such ideal for the given $\mathcal{D}$ we denote it by 
$\mathcal{B}:=\mathcal{B}_{\mathcal{D}}$.

\begin{definition}
	An ideal $\mathcal{B}$ is called maximally differential ideal if there is a set  of derivations $\emptyset\neq\mathcal{D}$  such that $\mathcal{B}=\mathcal{B}_{\mathcal{D}}$.
\end{definition}
By  $\Omega_{R/k}$ we mean the module of  K\"{a}hler  differential.   It is well-known that $\Der_k(R)=\Omega_{R/k}^\ast$. Let $I$ be an ideal of $R$ and set $\overline{R}:=R/ I$. We use the conormal sequence
$$0\lo \frac{I^{(2)}}{I^2}\lo I/I^2 \lo \Omega_{R/k}\otimes_R\overline{R}\lo \Omega_{\overline{R}/k}\lo 0,$$
where $I^{(2)}$ is the second symbolic power.
From now on we restrict the $k$-algebra $R$ be such that $\Der_k(R)$ is finitely generated as an $R$-module.
To see more informations on the module of differentials we recommend the  survey article  \cite{h}, and the related chapters  of the  books \cite{EIS} and \cite{mat}. Finally, we would like to cite the excellent
work \cite{PS} for prime characteristic methods in local algebra.

\section{Complete-intersection and differential ideals}

Here, we give negative answer to Question \ref {Q1} by presenting  several examples.

\begin{example}\label{em2}
	Let  $R:=\frac{\mathbb{F}_2[X,Y]}{(X^2,Y^2)}$. There is a derivation $D$ such that 
the following holds: 	
	\begin{enumerate}
		\item[i)] $(x,y)=\mathcal{B}:=\mathcal{B}_{D}$,
		\item[ii)] $\mathcal{B}$ is totally reflexive, i.e., $T_\bullet:=0\to \mathcal{B}\stackrel{\id}\lo \mathcal{B}\to 0$ is a minimal $G$-resolution of $\mathcal{B}$,
		\item[iii)] $F_n(T_\bullet)$ is acyclic for all $n\geq 1$,
		\item[iv)] $\mathcal{B}$ is not a complete intersection.
	\end{enumerate}
Also, there is a derivation $D'$ such that $\mathcal{B}_{D'}=(x)$.
\end{example}

\begin{proof}
i)	We define $D:=X\partial/\partial X+Y\partial/\partial Y$. We first check that for $D$ is zero
	over any of the relations: 
		\begin{enumerate}
		\item[1)] $D(X^2)=2XD(X)=0$,
		\item[2)]$D(Y^2)=2YD(Y)=0.$
	\end{enumerate}
So, $D$ induces a well defined derivation on  $R$. We use  the small letter for the residue class of elements in $R$. 
Note that $D(x)=x$ and $D(y)=y$. If we set $\fm:=(x,y)$ then $D(\fm)\subset\fm$.
Since $\fm$ is maximal, and in view of the definition, we see that $\fm$ is the maximally differential ideal with respect to $D$. So, $\mathcal{B}=(x,y)$.

ii) The ring $R$ is Gorenstein (in fact, complete-intersection) and is of
dimension zero.
It follows that any module is totally reflexive, as claimed.

iii) Any linear functor sends an isomorphism to an isomorphism.
From this  $F_n(T_\bullet)$ is acyclic for all $n\geq 1$.

iv) Neither $\fm=0$ nor generated by a regular sequence. By definition, $\mathcal{B}$ is not a complete intersection.

If we set $D':=x\partial/\partial x+\partial/\partial y$ then $\mathcal{B}_{D'}=(x)$.
\end{proof}

\begin{remark}\label{rem}
The ring is differentially simple with respect to all of derivations,
i.e.,  $\mathcal{B}_{\Der(R)}=0$.
	\end{remark}

One can prove that \begin{example}\label{e3}
	Let $R:=\frac{\mathbb{F}_p[X_1\ldots,X_n]}{(X^p_1\ldots,X_n^p)}$. There is a derivation $D$ such that: 	
	\begin{enumerate}
		\item[i)] $(x_1,\ldots,x_n)=\mathcal{B}:=\mathcal{B}_{D}$,
		\item[ii)] $\mathcal{B}$ is totally reflexive, i.e., $T_\bullet:=0\to \mathcal{B}\stackrel{\id}\lo \mathcal{B}\to 0$ is a minimal $G$-resolution of $\mathcal{B}$,
		\item[iii)] $F_n(T_\bullet)$ is acyclic for all $n\geq 1$,
		\item[iv)] $\mathcal{B}$ is not a complete intersection.
	\end{enumerate}
\end{example}

Here, the ring is  $1$-dimensional and the example is of any characteristic:

\begin{example}
 Let $k$ be any field  and let $R:=\frac{k[X,Y]}{(XY)}$. There is a derivation $D$ such that: 	
 \begin{enumerate}
 	\item[i)] $(x,y)=\mathcal{B}:=\mathcal{B}_{D}$,
 	\item[ii)] $\mathcal{B}$ is totally reflexive, i.e., $T_\bullet:=0\to \mathcal{B}\stackrel{\id}\lo \mathcal{B}\to 0$ is a minimal $G$-resolution of $\mathcal{B}$,
 	\item[iii)] $F_n(T_\bullet)$ is acyclic for all $n\geq 1$,
 	\item[iv)] $\mathcal{B}$ is not a complete intersection.
 \end{enumerate}
\end{example}

\begin{proof}
	i)	We define $D:=X\partial/\partial X-Y\partial/\partial Y$. We first check that for $D$ is zero
	over any of the relations: 
$$D(XY)=X\partial/\partial X(XY)-Y\partial/\partial Y(XY)=XY-YX=0.$$
So, $D$ induces a well defined derivation on  $R$.
	Note that $D(x)=x$ and $D(y)=y$. If we set $\fm:=(x,y)$ then $D(\fm)\subset\fm$.
	Since $\fm$ is maximal, $\mathcal{B}=(x,y)$.
	
	ii) The ring $R$ is Gorenstein (in fact, complete-intersection) and is of
	dimension one.
	It follows that any module is of $\G$-dimension at most one. In view of
	$0\to \mathcal{B}\to R\to R/\mathcal{B} \to 0$ we deduce that
	$\mathcal{B}$ is of $\G$-dimension zero. In other words, $\mathcal{B}$  is totally reflexive.
	
	iii) Any linear functor sends an isomorphism to an isomorphism.
	From this  $F_n(T_\bullet)$ is acyclic for all $n\geq 1$.
	
	iv) Neither $\fm=0$ nor $\fm$ is generated by a regular sequence. By definition, $\mathcal{B}$ is not a complete intersection.
\end{proof}

\begin{proposition}
Let $R$ be any zero-dimensional Gorenstein ring of prime characteristic $p$.
	 There is a derivation $D$ such that 
the following holds: 	
	\begin{enumerate}
		\item[i)] $\mathcal{B}$ is totally reflexive, i.e., $T_\bullet:=0\to \mathcal{B}\stackrel{\id}\lo \mathcal{B}\to 0$ is a minimal $G$-resolution of $\mathcal{B}$,
		\item[ii)] $F_n(T_\bullet)$ is acyclic for all $n\geq 1$,
		\item[iii)] $\mathcal{B}$ is not a complete intersection.
	\end{enumerate}
\end{proposition}

\begin{proof}
If $R$ is of the form $\frac{k[X_1\ldots,X_n]}{(X^p_1\ldots,X_n^p)}$
then we get the claim, see  Example \ref{e3}.
Then we may assume $R$ is not of that form. In particular,
$R$ is not differentially simple with respect to all of derivations.
From this, $\mathcal{B}$ is not zero with respect to any subset of derivations.
By adopting the proof of Example \ref{em2} we see all of claims are true.
\end{proof}

Here, we present  an example which is an integral domain:

\begin{example}\label{inte}
	Let $R:=\mathbb{F}_2[T^4,T^5,T^6]$ which is $1$-dimensional. There is a derivation $D$ such that 
	the following holds: 	
	\begin{enumerate}
		\item[i)] $(T^4,T^5,T^6)=\mathcal{B}:=\mathcal{B}_{D}$,
		\item[ii)] $\mathcal{B}$ is totally reflexive, i.e., $T_\bullet:=0\to \mathcal{B}\stackrel{\id}\lo \mathcal{B}\to 0$ is a minimal $G$-resolution of $\mathcal{B}$,
		\item[iii)] $F_n(T_\bullet)$ is acyclic for all $n\geq 1$,
		\item[iv)] $\mathcal{B}$ is not a complete intersection.
	\end{enumerate}
Also, there is another derivation $D'$ such that 
	\begin{enumerate}
	\item[a)] $(T^4,T^6)=\mathcal{B}:=\mathcal{B}_{D'}$,
	\item[b)] $\mathcal{B}$ is totally reflexive, i.e., $T_\bullet:=0\to \mathcal{B}\stackrel{\id}\lo \mathcal{B}\to 0$ is a minimal $G$-resolution of $\mathcal{B}$,
	\item[c)] $F_n(T_\bullet)$ is acyclic for all $n\geq 1$,
	\item[d)] $\mathcal{B}$ is not a complete intersection.
\end{enumerate}
\end{example}

\begin{proof}
Note that	$R=\frac{\mathbb{F}_2[X,Y,Z]}{(XZ-Y^2,X^3-Z^2)}$.
For the first part, we define $D:=Y\partial/\partial Y$. We need to check that  $D$ is zero
over any of the relations: 
\begin{enumerate}
	\item[1)] $D(XZ-Y^2)=YD(Y^2)=2Y^2=0$,
	\item[2)]$D(X^3-Z^2)=0.$
\end{enumerate}
So, $D$ induces a well defined derivation on  $R$.  
Note that $D(X)=D(Z)=0$ and $D(Y)=Y$.  We use  the small letter for the residue class of elements in $R$. If we set $\fm:=(x,y,z)$ then $D(\fm)\subset\fm$.
Since $\fm$ is maximal,  $\mathcal{B}=(x,y,z)$.

ii) The ring $R$ is Gorenstein (in fact, complete-intersection) and is of
dimension one.
It follows that any module is of $\G$-dimension at most one. In view of
$0\to \mathcal{B}\to R\to R/\mathcal{B} \to 0$ we deduce that
$\mathcal{B}$ is of $\G$-dimension zero. In other words, $\mathcal{B}$  is totally reflexive.

iii) Any linear functor sends an isomorphism to an isomorphism.
From this  $F_n(T_\bullet)$ is acyclic for all $n\geq 1$.

iv) Neither $\fm=0$ nor $\fm$ is generated by a regular sequence. By definition, $\mathcal{B}$ is not a complete intersection.

For the second part, we set $D':=\partial/\partial y$. Since $D'$ is zero
over any of the relations, i.e., it induces a well defined derivation on  $R$.
Note that $D'(X)=D'(Z)=D'(Y^2)=0$ and $D'(Y)=1$.
We combine this with the fact that $\ell(R/(x,y^2,z))=2$ to conclude
that   $\mathcal{B}=(x,y^2,z)$. Since $xz-y^2=0$ we deduce that $\mathcal{B}=(x,z)$, and we see the desired claims.
\end{proof}

\section{Derivations and the Gorenstein property}

By $\mu(-)$ we mean the minimal number of elements that needs to generate $(-)$.
Recall that the embedding dimension of a local ring $(R,\fm)$ is $\embdim R := \mu(\fm)$.
Now, we recall the following:

\begin{fact}\label{mat}(See \cite[Exercise 21.2]{mat})
 Let $A$ be a  local ring with $\embdim A = \dim A + 1$. If $ A$ is Cohen-Macaulay,
	then it is  complete-intersection.
\end{fact}
Here, we affirmatively  answer Question \ref {Q2}.

\begin{proposition}\label{mal}
	Maloo's ring is complete-intersection, and so Gorenstein.
\end{proposition}

\begin{proof}
Recall that Maloo's ring is a 1-dimensional noetherian local integral domain which is not regular.
In particular,
it is Cohen-Macaulay. By looking
at the example, we see its maximal ideal
generated by two elements. Since
$(R,\fm)$ is not regular, $\mu(\fm)=2$.
So, $\embdim R=\mu(\fm) =2= \dim R+ 1$. In view of Fact \ref{mat}, $R$ is complete-intersection.
\end{proof}

Here, we show both parts of Conjecture \ref{C1} are not true:

\begin{example}\label{z}
There is a  local $k$-algebra  $R$ for which the $R$-module $\Der_k(R)$ is finitely generated such that
\begin{enumerate}
\item[a)] $\Gdim(\Der_k(R))<\infty$,
\item[b)] $\CI(\Der_k(R))<\infty$,
\item[c)] $R$ is not Gorenstein, and so it is not complete-intersection.
\end{enumerate}
\end{example}

\begin{proof}
Let $R:=\frac{\mathbb{F}_2[X,Y]}{(X^4,X^2Y^2,Y^4)}$.	For simplicity, we set $k:=\mathbb{F}_2$.
We set $f_1:=X^4$, $f_2:=X^2Y^2$ and $f_3:=Y^4$. For any $f\in K[X,Y]$ we set $df:=\partial(f)/\partial Xdx+\partial(f)/\partial Ydy$.	From the conormal sequence, the module of  K\"{a}hler  differential is of the from $$\Omega_{R/k}\cong
 Rdx\oplus Rdy  /(df_1,df_2,df_3),$$ where  $Rdx\oplus Rdy$ is a free $R$-module with base $\{dx,dy\}$. Also, we have 
\begin{enumerate}
\item[1)] $df_1=\partial(X^4)/\partial Xdx+\partial(X^4)/\partial Y dy=4X^3dx=0$,
\item[2)] $df_2=\partial(X^2Y^2)/\partial Xdx+\partial(X^2Y)/\partial Ydy=2Y^2Xdx+2YX^2dy=0$,	
\item[3)] $df_3=\partial(Y^4)/\partial Xdx+\partial(Y^4)/\partial Y dy=4Y^3dy=0$.
\end{enumerate}
Thus, $\Omega_{R/k}$ is a free $R$-module of rank $2$. Note that $$\Der_k(R)=\Omega_{R/k}^\ast\cong\Hom_R(\Omega_{R/k},R)\cong R^2.$$
From this,  $$\Gdim(\Der_k(R))=\CI(\Der_k(R))=\pd(\Der_k(R))=0<\infty,$$i.e., $a)$ and $b)$ are valid. In order to show c), we compute the socle of $R$ and it is enough to show that it is not one-dimensional.
Indeed, $\Soc(R)=(0:_R\fm)\supseteq(x^3y,y^3x)$.
\end{proof}
We need the following result:

\begin{fact}(Jacobian criterion, see \cite[Corollary 16.22]{EIS}) \label{e}
Let $R$ be an  equi-dimensional affine ring (not necessarily reduced) over a perfect field $k$. The following assertions are true:
\begin{enumerate}
	\item[a)]	Suppose that  $\Omega_{R/k}$  is locally free over $R$ and $R$ is reduced or $\Char k = 0$, then $\Omega_{R/k}$ has rank $d := \dim R.$
\item[b)]  The module $\Omega_{R/k}$ is locally free over $R$ and of rank  $d$ iff $R_{\fp}$ is a regular local ring for each prime ideal
	$\fp$ of $R$.
\end{enumerate}
\end{fact}
There is a natural way to rescue Conjecture \ref{C1} by putting some restrictions, e.g., $k$ is of zero characteristic and may be $R$ is of finite type over $k$ and etcetera. In what follows we present series of cases for the validity of  Conjecture \ref{C1}.

\begin{proposition}
Let $(S,\fn,k)$ be a regular local ring containing $k$, of zero characteristic and of dimension $d\geq2$. Let $n>0$ be an integer. Let  $R:=\frac{S}{\fn^n}$.
The following are equivalent:
\begin{enumerate}
\item[a)] $\Gdim(\Der_k(R))<\infty$,
\item[b)] $n=1$,
\item[c)] $R$ is  Gorenstein.
\end{enumerate}
\end{proposition}

\begin{proof} $a) \Rightarrow b)$:
It is easy to see that $\Der_k(R)$ is finitely generated. Then, by Auslander-Bridger formula $\Gdim(\Der_k(R))=0$.
In particular,  $\Der_k(R)$ is (totally) reflexive. By a result of Ramras \cite{ram}, over the ring $R$ any reflexive module is free. Thus,
 $\Der_k(R)$ is free. In general, freeness of a dual module does not imply the freeness of a module.
 But, if a ring is of depth zero, this happens, see \cite[Lemma 2.6]{ram}. Since our ring
 is of depth zero, we have $\Omega_{R/k}$ is free. Since $\Char(R)=0$ and in view of Fact \ref{e}(i) it is of rank zero. In the light of Jacobian criterion $R$ is regular. This implies that $n=1$. 

$b) \Rightarrow c)$: This is trivial.

 $c) \Rightarrow a)$: Over Gorenstein local rings any finitely generated module is of finite $\G$-dimension.
\end{proof}

In the same vein, one can prove that:

\begin{proposition}\label{m3}
Let $(R,\fm)$ be a local   algebra  of finite type over a field $k$ of
zero characteristic such that $\fm^3=0$, $\mu(\fm)>1$  and that $\fm^2\neq(0:\fm)$.
	Then $\Gdim(\Der_k(R))<\infty$ iff $R$ is  Gorenstein.
\end{proposition}

The particular case of the next result  follows from Proposition \ref{m3}.
Here, we present a direct method to reprove it.
 
\begin{proposition} \label{m2}Let $R$ be a local algebra  of finite type  
	over a field $k$ of zero characteristic such that $\fm^2=0$. Then
$\Der_k(R)$ syzygetically is equivalent to some nontrivial copies of $k$, i.e., $\Syz_i(\Der_k(R))\cong \Syz_j(k^n)$. In particular, the	
following are equivalent:
	\begin{enumerate}
		\item[a)] $\Gdim(\Der_k(R))<\infty$,
		\item[b)] 	$R$ is  Gorenstein.
	\end{enumerate}
\end{proposition}

\begin{proof}  
Recall that $\Syz_i(-)$ stands for the $i$-th syzygy module of $(-)$.
First, assume that $\Syz_1(\Omega_{R/k})=0$. Then $\Omega_{R/k}$ is free. Since $\Char(R)=0$ and in view of Fact \ref{e}(i) its rank iz zero.
 This implies that $R$ is regular, and so $\Syz_1(k)=0$.
Then we may assume that $\Syz_1(\Omega_{R/k})\neq0$. Since we work
with the minimal free resolution, $\Syz_1(\Omega_{R/k})\subset \fm R^{\beta_0(\Omega_{R/k})}$. We apply $\Hom(-,R)$ to $$0\lo\Syz_1(\Omega_{R/k})\lo  R^{\beta_0(\Omega_{R/k})}\lo \Omega_{R/k}\lo 0,$$ and deduce the following exact sequence $$0\lo\Omega_{R/k}^\ast\lo  R^{\beta_0(\Omega_{R/k})}\stackrel{d}\lo \Syz_1(\Omega_{R/k})^\ast.$$Let $S:=\im(d)\subset\Syz_1(\Omega_{R/k})^\ast$.  

	\begin{enumerate}
	\item[Claim i)]Let $L$ be an $R$-module such that $\fm L=0$. Then $\fm \Hom_R(L,R)=0$.
 Indeed, let $f\in\Hom_R(L,R)$ and $r\in \fm$. Also, let $\ell \in L$. By definition, $rf(\ell)=f(r\ell)=f(0)=0$.
	\end{enumerate}
Recall that  $\fm \Syz_1(\Omega_{R/k})\subset \fm(\fm R^{\beta_0(\Omega_{R/k})})=0.$  In view of Claim i) we deduce that  $\fm S\subset \fm \Syz_1(\Omega_{R/k})^\ast=0.$ 
If $S=0$ then $\Omega_{R/k}^\ast\cong  R^{\beta_0(\Omega_{R/k})}$ is free.  Since our ring
is of depth zero and in view of \cite[Lemma 2.6]{ram}, we have $\Omega_{R/k}$ is free. In view of Jacobian criterion $R$ is regular.  Then, without loss of the generality we may and do assume that $S\neq 0$. In particular, $S$ has a $R/\fm$-module structure
and since the $R/ \fm$ is a field,   $S=\oplus_{n} k$ for some $n>0$. 
We have the following short exact sequence$$0\lo\Der_k(R)\lo  R^{\beta_0(\Omega_{R/k})}\lo \oplus_{n} k\lo 0\quad(\ast)$$
So, $\Der_k(R)\cong \Syz_1(k^n)$. Now we prove the particular case:

$a) \Rightarrow b)$:  It follows from  $\Syz_i(\Der_k(R))\cong \Syz_j(k^n)$   that $\Gdim(k)=\Gdim(\oplus k)=(i-j)+\Gdim(\Der_k(R))<\infty$. By the local-global principal,
$R$ is Gorenstein. So, $b)$ follows.

$b) \Rightarrow a)$: This implication always holds.
\end{proof}

By $\id(-)$ we mean the injective dimension. 
The previous observation leads us to:

\begin{proposition} Let $R$ be any artinian algebra of finite type over a field $k$ of zero characteristic. The	
	following are equivalent:
	\begin{enumerate}
		\item[a)] $\id(\Der_k(R))<\infty$,
	\item[b)]  $\pd(\Der_k(R))<\infty$,	
	\item[c)]	$R$ is  regular.
	\end{enumerate}
\end{proposition}

	\begin{proof}All of things behave well with respect to localization. Then, we may assume the ring is local.
		In view of Fact \ref{e} we may assume that $\Omega_k(R)\neq 0$. Since $R$ is of depth zero its follows that $\Omega_k(R)^\ast\neq 0$.

		$a) \Rightarrow b)$:  Any dual module is a submodule of a free
		module. Then, there is an exact sequence $$0\lo\Der_k(R)\stackrel{f}\lo R^n\lo\coker(f)\lo 0\quad(\ast)$$ Since 
		$\depth(R)=0$ we know $\id(\Der_k(R))=\depth(R)=0$. This implies that $(\ast)$ splits, i.e., $\Der_k(R)\oplus\coker(f)=R^n$. By the Krull–Remak–Schmidt  theorem, $\Der_k(R)$ is free. So, $\id(R)=0$. In particular, $R$ is Gorenstein. Over Gorenstein rings, finiteness of
		injective dimension is same as finiteness of projective dimension (see \cite[Theorem 2.2]{lv}).
		So, $\pd(\Der_k(R))<\infty$.

			$b) \Rightarrow c)$: By Auslander-Buchsbaum formula, $\pd(\Der_k(R))=0$.
		In view of  \cite[Lemma 2.6]{ram}, we see $\Omega_{R/k}$ is free. We apply Jacobian criterion to deduce that $R$ is regular.
		
		$c) \Rightarrow a)$: This is trivial.	
\end{proof}
\begin{question}\label{id}
Let $(R,\fm)$ be an algebra  of finite type over a field $k$ of zero characteristic. Find situations for which the following are equivalent:
	\begin{enumerate}
	\item[a)]  $\id(\Der_k(R))<\infty$,
	\item[b)]  $\pd(\Der_k(R))<\infty$,	
	\item[c)]	$R$ is  regular.
\end{enumerate}
\end{question}

\begin{remark}
i) The existence of finitely generated module of finite injective dimension implies that
	the ring is Cohen-Macaulay.

ii)  Let $R$ be the Maloo's ring and recall that
$R$ is a non-regular ring of zero characteristic such that $\Der_k(R)$ is free.
In view of
Proposition \ref{mal} we see $R$ is Gorenstein. So, any free module of finite rank is of finite injective
dimension. In fact, $\id(\Der_k(R))=1<\infty$. But,
$R$ is not regular.

iii)  Recall that Maloo's ring is not of finite type over $k$. Then part ii) suggests that in Question \ref{id} we need to assume $R$ is of finite type
over $k$.

 iv)   Question \ref{id} is not true over rings of prime characteristic. Indeed, we look  at $R:=\mathbb{F}_2[X]/(X^2)$.
Then  $\Der_k(R)$ is free, and since $R$ is Gorenstein,  $\id(\Der_k(R))=0<\infty$. But,
$R$ is not regular.

 v)  Let $k$ be an algebraically closed field of zero characteristic,  $G$ be a group acts on $k$-vector space
$kX\oplus kY$ without pseudo-reflection. Let $R:=k[[X,Y]]^G$.  Then $\id(\Der_k(R)^\ast)<\infty$ iff  $R$ is  regular.
Indeed, suppose $\id(\Der_k(R)^\ast)<\infty$. The main idea is to use the fundamental sequence	$$0\lo\omega_R\lo A\lo R\lo k\lo 0 \quad(\ast), $$
where $A$ is a finitely generated module.	Under the stated assumptions, it is shown in \cite[Page 180]{martin} that $A\cong\Der_k(R)^\ast$.  In view of $(\ast)$ we have
$0\to\omega_R\lo \Der_k(R)^\ast\to \fm\to 0$. Since  
$\id(\omega_R)<\infty$ we deduce that $\id(\fm)<\infty$. We left to the reader 
to check that $R$  is regular (see \cite[Page 318]{lv}).

 v)  Adopt the notation of Question \ref{id} and suppose the ring is reduced and of finite type over $k$.  The conjectural implication $b) \Rightarrow c)$ was asked independently by Herzog and Vasconcelos. 
It may be worth to note that
this implication is a generalization of the  Zariski-Lipman conjecture, see \cite{lip}.
\end{remark}

A ring  $R$ is called almost complete intersection if there is a regular ring $S$ with an ideal $I$ such that
$\mu(I)\leq \Ht(I)+1$ and that $R=S/I$. This is slightly different from \cite{kunz}.

\begin{observation}
	\label{almost}
	Let $R$ be a $1$-dimensional almost complete intersection local domain which is essentially
	of finite type over a field  $k$ of zero characteristic. Suppose $\Der_k(R)$ is nonzero.	The following are equivalent:
	\begin{enumerate}
		\item[a)]  $\id(\Der_k(R))<\infty$,
		\item[b)]  $\pd(\Der_k(R))<\infty$,	
		\item[c)]	$R$ is  regular.
	\end{enumerate}
\end{observation}

\begin{proof} 	
Since $R$ is homomorphic image of a regular ring, it has a canonical module. Also, $R$ is Cohen-Macaulay.
Then the canonical module is denoted by $\omega_R$. Since $\Der_k(R)=\Hom_{R}(\Omega_{R/k},R)$ is a dual module, its depth is
one. By definition, $\Der_k(R)$ is maximal Cohen-Macaulay.

	$a) \Rightarrow b)$: 
 Maximal  Cohen-Macaulay modules of finite injective dimension
are of the form  $\omega_R^{\oplus n}$ for some $n>0$. A ring is called  quasi-reduced if it is $(S_1)$ and generically Gorenstein.
For example, reduced rings are quasi-reduced.   Over quasi-reduced rings any dual module is reflexive, see e.g. \cite[Fact 5.13]{moh} and references therein.
Since, $R$ is quasi-reduced and $\Der_k(R)$ is a dual module, we deduce that $\Der_k(R)$ is reflexive. Direct summand of reflexive modules is again reflexive.
Since $\Der_k(R)\cong\omega_R^{\oplus n}$, it follows that $\omega_R$ is reflexive. Now, we recall the following result of Kunz \cite[Theorem 1]{kunz}:
	\begin{enumerate} 	
			\item[Fact:] Suppose $A$ is almost complete-intersection. The following are equivalent
		\begin{enumerate}
	\item[i)]   $K_A$ is reflexive
	\item[ii)]  For all $\fp\in\Spec(A)$ with $\Ht(\fp) = 1$ the local ring $A_{\fp}$ is a complete
	intersection.
\end{enumerate}
\end{enumerate}
We apply this fact to see $R$ is complete-intersection. In particular,
$R$ is Gorenstein.  Over Gorenstein rings, finiteness of
injective dimension is same as finiteness of projective dimension.
So, $\pd(\Der_k(R))<\infty$.

$b) \Rightarrow c)$:  Recall that $\depth(\Der_k(R))=1$. By Auslander-Buchsbaum formula,
$\Der_k(R)$ is free. In the light of Lipman's result, $R$ is normal, see \cite[Theorem 1]{lip}.
Since $R$ is  $1$-dimensional, $R$ is regular.

$c) \Rightarrow a)$: This is trivial.
\end{proof}

\begin{example}Let $R:=k[T^3,T^4,T^5]$ where $k$ is a field of zero characteristic. Then $\id(\Der_k(R))=\infty=\pd(\Der_k(R))$.
\end{example}

\begin{proof}It is easy to see that
	$R\cong \frac{k[X,Y,Z]}{(X^{2}Y-Z^{2},XZ-Y^{2},YZ-X^{3})}$.
	By definition, $R$ is 	almost complete-intersection.
	Since $R$ is not regular, $\id(\Der_k(R))=\infty=\pd(\Der_k(R))$.  
\end{proof}
A ring  $R$ is called purely almost complete intersection if there is a regular ring $S$ with an ideal $I$ such that
$\mu(I)=\Ht(I)+1$ and that $R=S/I$.

\begin{corollary}Let $R$ be a $1$-dimensional purely almost complete intersection local domain which is essentially
	of finite type over a field  $k$ of zero characteristic. Suppose $\Der_k(R)$ is nonzero.	Then  $\id(\Der_k(R))=\pd(\Der_k(R))=\infty$.
\end{corollary}

\begin{proof} 	
By a result of 	Kunz \cite{kunz2}, purely almost complete intersection rings are not Gorenstein. Now,
the desired claim is in Observation \ref{almost}.
\end{proof}

The next result is an extension of Proposition \ref{m2}:

\begin{proposition}
	Let $R$ be a graded algebra over a field $k$ of zero characteristic. Assume $R$ is Cohen-Macaulay  and of minimal multiplicity. The following are equivalent:
	\begin{enumerate}
		\item[a)] $\Gdim(\Der_k(R))<\infty$,
		\item[b)]  $R$ is  Gorenstein.
	\end{enumerate}
\end{proposition}

\begin{proof}
Assume  $\Gdim(\Der_k(R))<\infty$. Suppose on the way contradiction that
$R$ is not Gorenstein. Apply this along with the Cohen-Macaulay  and minimal multiplicity assumptions to deduce 
that any totally reflexive module is free. For more details see e.g., \cite[Observation 4.17]{moh}.
Thus, $\pd(\Der_k(R))<\infty$. By a result of Platte \cite{p} (also see \cite[Theorem 4.5]{h}), $R$ is quasi-Gorenstein. Since $R$  is Cohen-Macaulay, we deduce that $R$ is  Gorenstein.
This completes the proof of  $a) \Rightarrow b)$. The reverse implication is trivial.
\end{proof}

\begin{example}Let $R:=k[T^3,T^4,T^5]$ where $k$ is a field of zero characteristic. Then $\Gdim(\Der_k(R))=\infty$.
\end{example}

\begin{proof}	It is easy to see that $R/T^3R\cong\frac{k[X,Y]}{\fm^2}$. Thus, $R$ is Cohen-Macaulay  and of minimal multiplicity.
	Since $R$ is not Gorenstein, $\Gdim(\Der_k(R))=\infty$.
\end{proof}

Let us compute the previous observations in a prime characteristic case:

\begin{example}
	Let $S:=\mathbb{F}_2[X,Y]$ and set $R:=\frac{S}{(X,Y)^2}$.
	Then $\Gdim(\Der_{\mathbb{F}_2}(R))=\infty=\id(\Der_{\mathbb{F}_2}(R))$.
\end{example}

\begin{proof}  For simplicity, we set $k:=\mathbb{F}_2$.
	We set $f_1:=X^2$, $f_2:=Y^2$ and $f_3:=XY$.	Recall that $\Omega_{R/k}\cong
	Rdx\oplus Rdy  /(df_1,df_2,df_3) $  where
	\begin{enumerate}
		\item[1)] $df_1=\partial(X^2)/\partial Xdx+\partial(X^2)/\partial Y dy=2Xdx=0$,
		\item[2)]$df_2=\partial(Y^2)/\partial Xdx+\partial(Y^2)/\partial Y dy=2Ydx=0$,
		\item[3)] $df_3=\partial(XY)/\partial Xdx+\partial(XY)/\partial Ydy=Ydx+Xdy$.
	\end{enumerate}
Since  $(df_3)R=\{(ry,rx):r\in R\}\cong R/(0:_R(y,x))= R/\fm$, there is an exact sequence $$0\lo R/\fm \lo R^2\lo\Omega_{R/k} \lo 0.$$ Recall that $\Hom_R(R/\fm,R)=\{r\in R:r\fm=0\}=\fm$.
	Dualizing this yields that $$0\lo\Der_k(R) \lo R^2\lo\fm\lo\Ext^1_R(\Omega_{R/k} ,R) \lo 0\quad(+)$$
	Now, we compute the projective resolution of $\Omega_{R/k}$:
	$$\ldots\lo R^2 \stackrel{( y,x)}\lo R\stackrel{(^y_x)}\lo R^2\lo \Omega_{R/k}  \lo 0.$$
	Apply $\Hom_R(-,R)$ to it, yields that  $$ \Ext^1_R(\Omega_{R/k} ,R)=H\left( R^2 \stackrel{(y,x)}\lo R\stackrel{(^y_x)}\lo R^2\right)=0.$$
	Putting this in $(+)$, we have the following exact sequence:$$0\lo\Der_k(R) \lo R^2\lo\fm \lo 0\quad(\ast)$$
	Suppose on the way of contradiction  $\Gdim(\Der_k(R))<\infty$. From $(\ast)$ we deduce that
	$\Gdim(\fm)<\infty$. 
	In view of $0\to\fm \to R\to  R/\fm \to 0$  we deduce that $\Gdim(k)<\infty$. It follows from the local-global principal that $R$ is Gorenstein. So,
	its socle is one-dimensional. On the other hand,
	$\Soc(R)=(0:_R\fm)=\fm=(x,y)$ is $2$-dimensional. This contradiction shows  that   $\Gdim(\Der_k(R))=\infty$.
	
	Suppose on the way of contradiction that $\id(\Der_k(R))<\infty$. Then it is injective. In view of $(\ast)$ we have  $ \Der_k(R)\oplus\fm \cong R^2$. The left  hand side has at least three generators and
	the right hand side needs two elements, a contradiction.
\end{proof}

Here, we work with a non-Cohen-Macaulay ring:

\begin{example}Let $R:=\frac{\mathbb{F}_2[X,Y]}{(X^2,XY^2)}$.
	Then $\Gdim(\Der_{\mathbb{F}_2}(R))=\id(\Der_{\mathbb{F}_2}(R))=\infty$.
\end{example}

\begin{proof}  For simplicity, we set $k:=\mathbb{F}_2$.
	It is easy to see that $\Omega_{R/k}\cong
	Rdx\oplus Rdy  /(Y^2dx).$  
	Note that  $(Y^2dx)R=\{(ry^2,0):r\in R\}\cong R/(0:_R(y^2))= R/ x R$.
	There is an exact sequence $$0\lo R/ xR \lo R^2\lo\Omega_{R/k} \lo 0\quad(\dagger)$$
Note that
	$\ldots\lo R  \stackrel{(x)}\lo R\stackrel{(^{y^2}_0)}\lo R^2\lo \Omega_{R/k}  \lo 0$ 	is the projective resolution of $\Omega_{R/k}$.
	Apply $\Hom_R(-,R)$ to it, yields the following complex
	$$0\lo\Der_k(R) \lo R^2\stackrel{(^{y^2}_0)}\lo R\stackrel{x}\lo R \lo \ldots\quad(+)$$
	Note that $\im((^{y^2}_0))=y^2R=R/(0:_Ry^2)=R/xR$ and $\ker(x)=(x,y^2)$. So, $$  \Ext^1_R(\Omega_{R/k} ,R)=H\left( R^2 \stackrel{(y^2,0)}\lo R\stackrel{(x)}\lo R\right)=\frac{(x,y^2)}{(y^2)}\neq0 \quad(\ast)$$  Since $(+)$ is left-exact, we have the following exact sequence
	$$0\lo\Der_k(R) \lo R^2 \lo\im((^{y^2}_0))= R/xR\lo 0\quad(\ast,\ast)$$
	Suppose on the way of contradiction  $\Gdim(\Der_k(R))<\infty$. Note that $\depth(R)=0$.
	By Auslander-Bridger formula, $\Der_k(R)$ is totally reflexive. 
	From $(\ast,\ast)$ and Auslander-Bridger formula, we deduce that
	$\Gdim(R/xR)=0$. In view of $(\dagger)$,  and Auslander-Bridger formula, we deduce that $\Omega_{R/k}$ is totally reflexive. Thus, $\Ext^1_R(\Omega_{R/k} ,R)=0$. This is in contradiction with  $(\ast)$. We proved that  $\Gdim(\Der_k(R))=\infty$.
	
	In view of $(\ast,\ast)$ we deduce that $\id(\Der_k(R))=\infty$.
\end{proof}

A  $1$-dimensional  local integral domain is called rigid if it admits no infinitesimal deformations. In other words, $T^1_R=0$. Our interests comes from $T^1_R\cong\Ext^1_R(\Omega_{R/k} ,R)$.
The  rigidity conjecture  says that rigid rings of zero characteristic are regular. For more details see  \cite{h}.

\begin{proposition}
	Let $R$ be a $1$-dimensional local irreducible $k$-algebra over a  field $k$ of zero characteristic. Assume $R$ is  of minimal multiplicity. The following are equivalent:
	\begin{enumerate}
		\item[a)] $\Ext^i_R(\Omega_{R/k} ,R)=0$ for all $1 \leq i \leq 4$,
		\item[b)]  $R$ is  regular.
	\end{enumerate}
\end{proposition}

\begin{proof}Since $R$ is reduced, it satisfies Serre's condition $(S_1)$. In particular, $R$ is Cohen-Macaulay.

	$a) \Rightarrow b)$: First, we show that $R$ is Gorenstein. 
	Suppose on the way of contradiction that $R$ is not Gorenstein, and search for a contradiction. We are in the situation of \cite[Observation 4.17]{moh}. It says that over 
a	Cohen-Macaulay ring  of minimal multiplicity which is not Gorenstein,
	 the vanishing of  $\Ext^i_R(\Omega_{R/k} ,R)=0$ for   $1 \leq i \leq 4$ implies  that $\Omega_{R/k}$ is free. One may use Jacobian criterion to see  $R$ is regular (or use the Lipman's result).
	But, regular rings are Gorenstein. This contradiction allow us to assume $R$ is Gorenstein. In particular,  the rigidity conjecture is equivalent with the  Berger's conjecture, see \cite{ber} for the statement. We are going to prove  the latter one. Since $R$ is of minimal multiplicity, there is a parameter element $r$ such that $r\fm=\fm^2$. Set $\overline{R}:=R/rR$. Then $(\overline{R},\fn)$
	is a zero-dimensional Gorenstein ring with $\fn^2=0$. Without loss of the generality we may and do assume that $\fn\neq 0$. Then $\Soc(R)=\fn$. Since the socle is one-dimensional, there is some $s\in\fm$ such that $\fn=(\overline{s})$, where $\overline{s}:=s+rR$. Since $\fn=\frac{\fm}{rR}$ we conclude that $\fm=(s,r)$. In view of Fact \ref{mat} $R$ is complete-intersection. But, Berger's conjecture is valid in this case, and so $R$ is regular, as claimed.
	
	$a) \Rightarrow b)$: This is trivial because $\Omega_{R/k}$ is free.
\end{proof}

\section{Differentials and tensor products}

Recall that
\begin{question}(See  \cite[Question   3.7]{j}) Let $R$ be a local ring and let $M$ be a finite module. Does there exist a torsionless
	$	R$-module $U \neq 0$ such that $$\depth(M  \otimes_R U )= \depth(N  \otimes_R U )$$
	whenever $N$ is a non-zero finite $R$-module with $\depth M \leq \depth N$?
\end{question}

\begin{observation}
	Let $(R,\fm,k)$ be a  local integral domain of depth $d>0$ and let $M:=\fm$ and let $N:=R$.  Suppose there is a torsionless
	$	R$-module $U \neq 0$ such that $$\depth(M  \otimes_R U )= \depth(N  \otimes_R U )\quad(\ast)$$
 Then 	 $\depth M \leq \depth N$ and $d=1$.
\end{observation}

\begin{proof}
Thanks to $(\ast)$,  we know $\depth(M  \otimes_R U )= \depth(N  \otimes_R U )=\depth(U)>0$, because $U$  is  torsionless and $R$ is an integral domain. We  look at the short exact sequence $\zeta:=0\to M \to R\to k\to 0$. We apply the functor $U\otimes_R-$ to  $\zeta$ and obtain the following exact sequence $$0=\Tor^R_1(R,U)\lo \Tor^R_1(U,k)\lo M\otimes_R U\lo U\lo k\otimes_R U\lo 0.$$ On one hand $\Tor^R_1(U,k)\subset M\otimes_R U$ and $\ell(\Tor^R_1(U,k))<\infty$. On the other hand, $\depth(M  \otimes_R U )>0$. Combining these, $\Tor^R_1(U,k)=0$. From this, $\pd(U)=0$. According to a folklore result of Kaplansky, $U$ is free. In view of $(\ast)$, we observe $$\depth(\fm)=\depth(M  \otimes_R U )= \depth(N  \otimes_R U )=\depth(R)=d\quad(+)$$Also, $\zeta$
induces the exact sequences  $0\to H^0_{\fm}(\fm)\to H^0_{\fm}(R)=0$  and  $0=H^0_{\fm}(R)\to H^0_{\fm}(k)=k\to H^1_{\fm}(\fm).$ Thus,
$\depth(\fm)=\inf\{i:H^i_{\fm}(\fm)\neq 0\}=1$. Apply this along with $(+)$ to see $d=1$, as claimed.	
\end{proof}

\begin{corollary}
Concerning the above observation, take $d>1$. This yields a negative answer to Question 5.1.
\end{corollary}
The following is in \cite[Proposition 3.14]{j} under the extra assumption
 $\Tor^R_4 (\Der_k(R),R/\fm^r) =0$.

\begin{fact} 	Let $(R,\fm)$ be a local $k$-algebra with $\depth R \geq 1$ , where $k$ is a ring such
that $\Der_k(R)$ is a finite $ R$-module. Suppose there exists a torsionless $R$-module $U \neq  0 $ such
that $\depth( \Der_k(R) \otimes_R U)\neq\depth(U)-1$.
 Assume
$\Tor^R_3 (\Der_k(R), R / \fm^r) = 0$ for some integer $ r \geq 1$.
Then $\Der_k(R)$ is   free.
\end{fact}

\begin{proof}
By a result of Celikbas and Takahashi \cite[Corollary 1.3]{c}, 	${\fm^r}$
 is $\Tor$-rigid. So, $\frac{R}{\fm^r}$ is as well. This implies that  $\Tor^R_4 (\Der_k(R), \frac{R}{\fm^r}) = 0$. Now, we apply \cite[Proposition 3.14]{j}.
\end{proof}

The following is in \cite[Proposition 4.6]{j} under the extra assumption $\Ext^d_R (\Der_k(R), \Der_k(R)) = 0$.

\begin{fact} 	Let $(R,\fm)$ be a Gorenstein complete local $k$-domain of dimension $d \leq  3$, where $k$ is
	a field such that $ k\subset R$ is a finite extension. Assume that $R$ is locally a complete intersection
	in codimension $1$ (e.g., R normal). If $\Ext^1_R (\overline{\Omega}_{R/k}, R) = 0$ and $\Ext^i_R (\Der_k(R), \Der_k(R)) = 0$ for all $1\leq i\leq d-1$\footnote{ If the ring is of isolated singularity, one needs vanishing of very few of these Ext-modules.},  then $\Der_k(R)$ is free.
\end{fact}

\begin{proof}
Combine \cite[Theorem]{o} and \cite[Proposition 4.6]{j}.
\end{proof}

The following was asked in \cite[Remark 3.5]{j}:

\begin{question}
	Let $R$ be a local $k$-algebra which is integral domain  $\Der_k(R)$ is finitely generated of finite projective dimension and that $\depth(R)\leq 3$.  What is $\depth(\Der_k(R)\otimes_R\Der_k(R))$?
\end{question}

\begin{observation}\label{deptht}
	Adopt the above notation and suppose in addition $R$ is $3$-dimensional and Cohen-Macaulay. Then 
	$\depth(\Der_k(R)\otimes_R\Der_k(R))=1$ if and only if $\Der_k(R)$ is not free.
\end{observation}

\begin{proof}Suppose $\Der_k(R)$ is not free. 
	Recall that $\Der_k(R)$ is a dual module.  Thus, $\depth(\Der_k(R))\geq 2$. In the light of Auslander-Buchsbaum formula we observe that $\pd(\Der_k(R))\leq 1$. Without loss of generality we may assume that $\pd(\Der_k(R))= 1$, because   $\Der_k(R)$ is not free. By the same reasoning, 
	$\Der_k(R)$ is locally free over the punctured spectrum.
	Now, recall
	from \cite[Proposition 6.2]{finitsup} that 	 $\depth(\Der_k(R)\otimes_R\Der_k(R))= \depth(R)-2=1$. The converse part is trivial.
\end{proof}

We cite to \cite{SCH} as an additional reference.


\begin{thebibliography}{99}
 \bibitem{moh}
M. Asgharzadeh, \emph{Reflexivity revisited}, arXiv:1812.00830   [math.AC].


\bibitem{finitsup}
M. Asgharzadeh,
\emph{Finite support of tensor products}, arXiv:1902.10509   [math.AC].


 \bibitem{AB} M. Auslander and M. Bridger, \emph{Stable module theory}, Mem. of
 the AMS  {\bf94}, Amer. Math. Soc., Providence 1969.

 \bibitem{ber} R.W. Berger, \emph{Differential moduln eindimensionaler lokaler ringe}, Math. Z. {\bf81}  (1963), 326–354. 


\bibitem{c} O.
Celikbas and R. Takahashi,  \emph{Powers of the maximal ideal and vanishing of (co)homology}, Glasg. Math. J. {\bf63} (2021), no. 1, 1–5.

 \bibitem{EIS}D.
Eisenbud, \emph{Commutative algebra. With a view toward algebraic geometry}, Graduate texts in mathematics, {\bf150} Springer-Verlag, New York, 1995. 

\bibitem{j}V.H.
Jorge-Pérez and C.B. Miranda-Neto,  \emph{Homological aspects of derivation modules and critical case of the Herzog–Vasconcelos conjecture},  Collectanea Math, to appear.
\bibitem{h}
J. Herzog, \emph{The module of differentials}, Workshop on commutative algebra and its relation to combinatorics and computer algebra, Trieste,
1994.

\bibitem{kunz}E.
Kunz, \emph{The conormal module of an almost complete intersection}, Proc. AMS {\bf 73} (1979),  15–21. 


\bibitem{kunz2}E.
Kunz, \emph{Almost complete intersections are not Gorenstein rings}, J. Algebra {\bf 28} (1974), 111–115.

\bibitem{lv}
G.L. Levin and Wolmer V. Vasconcelos, \emph{Homological dimensions and Macaulay rings}, Pacific J. Math.,
{\bf25}  (1968), 315-323.

\bibitem{lip}
J. Lipman, \emph{
	Free derivation modules},
Amer. J. Math.,  {\bf87} (1965), 874-898.

 \bibitem{mal}
 A.K. Maloo,
 \emph{Differential simplicity and the module of derivations},
 J. Pure Appl. Algebra, {\bf115} (1997),   81-85.
 
 
  \bibitem{martin}A.
Martsinkovsky,  \emph{Almost split sequences and Zariski differentials}, Trans. AMS
 {\bf319}, 285-307 (1990). 
 
 
 
 
 
 \bibitem{mat}
 H. Matsumura, \emph{Commutative ring theory}, Cambridge Studies in Advanced Math, \textbf{8}, (1986).
 
 	\bibitem{int}C.B.
Miranda-Neto,  \emph{Maximally differential ideals of finite projective dimension}, Bull. Sci. Math. {\bf{ 166}} (2021), 102936. 


	\bibitem{o}   M. Ono and Y. Yoshino, \emph{An Auslander-Reiten principle in derived categories}, J.
	Pure Appl. Algebra {\bf{221}} (2017), no. 6, 1268–1278.

	\bibitem{p} E.
Platte, \emph{Zur endlichen homologischen dimension von differential moduln},
Manuscripta Math. {\bf32} (1980),  295–302.

\bibitem{ram}
M. Ramras, \emph{Betti numbers and reflexive modules}, Ring theory (Proc. Conf., Park City, Utah, 1971),   297-308. Academic Press, New York, 1972.

\bibitem{SCH}
G.
Scheja and U. Storch, \emph{Differentielle eigenschaften der lokalisierungen analytischer algebren},
Math. Ann. {\bf197} (1972), 137–170.

\bibitem{PS}
C. Peskine and L. Szpiro, \emph{Dimension projective finie et cohomologie locale},
Publ. Math. IHES.  {\bf42} (1973),  47--119.

\bibitem{Y}S. Yuan, \emph{Differentiably simple rings of prime characteristic}, Duke Math. J. {\bf{31}} (1964) 623–630.




 \end{thebibliography}
\end{document}